\documentclass[a4paper,11pt]{article}

\usepackage[T1]{fontenc}
\usepackage[utf8]{inputenc}
\usepackage{lmodern}
\usepackage[english]{babel}
\usepackage{microtype}

\usepackage[a4paper, margin=1in]{geometry}

\usepackage{amsmath, amssymb, amsfonts, amsthm}
\usepackage{mathtools}
\usepackage{thmtools}
\usepackage{bm}
\usepackage{bbm}
\usepackage{mathrsfs}
\usepackage{physics}
\usepackage{tikz-cd}
\usepackage{tikz}
\usetikzlibrary{decorations.markings}
\usetikzlibrary{arrows.meta}
\usepackage{enumitem}
\usepackage{stmaryrd}
\usepackage{caption}
\usepackage{xcolor}
\usepackage{graphicx}

\usepackage{hyperref}
\usepackage{cleveref}

\definecolor{colorlinks}{RGB}{0, 24, 168}
\definecolor{colorcites}{RGB}{124, 10, 2}
\hypersetup{
  colorlinks=true,
  linkcolor=colorlinks,
  citecolor=colorcites,
  urlcolor=colorlinks,
  pdfborder={0 0 0}
}

\newtheorem{theorem}{Theorem}[section]

\newtheorem{lemma}[theorem]{Lemma}
\newtheorem{proposition}[theorem]{Proposition}
\newtheorem{corollary}[theorem]{Corollary}

\theoremstyle{definition}
\newtheorem{definition}[theorem]{Definition}

\newtheorem{remark}[theorem]{Remark}

\newcommand*{\nombres}[1]{\mathbb{#1}}

\newcommand*{\ZZ}{\nombres{Z}}

\newcommand*{\RR}{\nombres{R}}
\newcommand*{\CC}{\nombres{C}}

\newcommand*{\ldef}{\coloneqq}

\newcommand*{\term}[1]{\emph{#1}}
\newcommand*{\ie}{i.e.}

\DeclareMathOperator{\re}{Re}

\title{The zeta function of regular trees, their special values and functional equations}
\date{March 11, 2026}
\author{Dylan Müller}

\begin{document}
\maketitle

\begin{abstract}
We determine the special values at positive integers of the spectral zeta function associated with the combinatorial Laplacian on the regular tree. These values admit explicit formulas in terms of certain polynomials, which we show to be palindromic and to have non-negative integer coefficients with a combinatorial interpretation. Along the way, we uncover unexpected symmetries between the values of the zeta function at negative and positive integers, expressed at the level of their generating functions. Using these symmetries, we ultimately establish a functional equation of the type \( s \longleftrightarrow 1-s \) for a natural completion of the zeta function.
\end{abstract}

\section{Introduction}

In 1735, Euler famously found the exact value of the series
\begin{equation}\label{Riemann: Special value even}
    \zeta(2k) \ldef \sum_{n \ge 1} n^{-2k}
    = \frac{(2\pi)^{2k}}{2(2k)!}(-1)^{k+1} B_{2k},
\end{equation}
where \( B_{2k} \) denotes the Bernoulli numbers.
Later, he computed the values at negative odd integers
\[
    \zeta(1-2k) = -\frac{B_{2k}}{2k},
\]
and emphasized the striking symmetry between positive and negative special values.
This phenomenon ultimately led to the functional equation of the Riemann zeta function, later established by Riemann in his seminal 1859 memoir, where the zeta function was defined over the complex plane and its fundamental analytic properties were proven.

The general philosophy that symmetries manifest themselves through special values motivates the present work. We investigate the values at integers of the spectral zeta function associated with the regular tree \( T_{q+1} \) of degree \( q+1 \), which is the universal cover of every \((q+1)\)-regular graph. We find that these special values enjoy unexpectedly rich algebraic and combinatorial structure, together with two remarkable symmetries that ultimately lead to a functional equation for the zeta function of the tree.

\subsection*{Spectral zeta function of the tree}

In full generality, the spectral zeta function of a space \( X \) equipped with a Laplacian is defined as the Mellin transform of the trace of the heat kernel
\[
\zeta_X(s) \ldef \frac{1}{\Gamma(s)}\int_0^\infty \mathrm{Tr}(K_X(t))\, t^s \,\frac{\mathrm{d}t}{t},
\]
discarding the \( 0 \)-eigenvalue when it belongs to the point spectrum of the Laplacian. For instance, for the circle \( \RR/\ZZ \), whose spectrum is given by \( \lambda = 4 \pi^2 n^2 \), with \( n \in \ZZ \), one obtains
\[
 \zeta_{\RR/\ZZ}(s) = \sum_{\lambda \neq 0}\lambda^{-s} = 2 (2\pi)^{-2s}\zeta(2s).
\]

Another fundamental example is the spectral zeta function of the discrete line \( \ZZ \):
\begin{equation*}
    \zeta_{\ZZ}(s)
    = \frac{1}{\Gamma(s)}\int_0^{\infty} e^{-2t}I_0(2t)\,t^s \,\frac{\mathrm{d}t}{t}
    = \frac{1}{2^{2s}\pi^{1/2}} \frac{\Gamma(1/2-s)}{\Gamma(1-s)}
    = \binom{-2s}{-s}
    = \prod_{k = 1}^{\infty}\frac{(1-s/k)^2}{(1-2s/k)}.
\end{equation*}
Its special values at negative integers are central binomial coefficients and are directly related to Catalan numbers. Quite remarkably, it admits a functional equation of the type \( s \leftrightarrow 1-s \), in complete analogy with other zeta functions in the continuous setting, and shares with them many other properties; see \cite{FK17, KP23}.

As a more general graph, the regular tree \( T_{q+1} \) of degree \( q + 1 \) is fundamental among regular graphs. Through the spectral measure, its spectral zeta function is given by
\begin{equation}\label{Def:zeta function on the tree}
    \zeta_q(s) = \frac{q+1}{2\pi}\int_{-2\sqrt{q}}^{2\sqrt{q}} (q+1-\lambda)^{-s}\frac{\sqrt{4q-\lambda^2}}{(q+1)^2-\lambda^2}\,\mathrm{d}\lambda.
\end{equation}
It is a special case of a hypergeometric function; see \cite[Theorem 1.4]{FK17}. For \( q > 1 \),
\begin{equation}\label{Eq:FK17:Appell formula}
    \zeta_{q}(s)=\frac{q(q+1)}{(q-1)^{2}(\sqrt{q}-1)^{2s}}F_{1}(3/2,s+1,1,3;u,v),
\end{equation}
with \( u = -4\sqrt{q} / (\sqrt{q}-1)^{2} \), \( v = 4\sqrt{q} / (\sqrt{q}+1)^{2} \),
and where \( F_1 \) is one of Appell's hypergeometric functions.

\subsection*{Main results}

One of our main results is the determination of the special values at positive integers of \( \zeta_q \) through a family of polynomials with remarkable algebraic and combinatorial properties.

\begin{theorem}\label{Thm:Main}
    For all integers \( n \ge 1 \) and \( q \ge 2 \), we have
    \begin{equation}\label{Def:P_n}
        \zeta_q(n) = \frac{q}{(q-1)^{2n-1}(q+1)^n}P_n(q),
    \end{equation}
    where \( P_n \) is a palindromic and monic polynomial of degree \( 2n-2 \) with integer coefficients. Moreover, the coefficients of \( P_n \) count weighted $2$-coloured Dyck words, and hence are non-negative.
\end{theorem}

The first polynomials can be computed by hand:
\begin{table}[h!]
\centering
\begin{tabular}{c|c|c|c|c}
$n$
    & $1$
    & $2$
    & $3$
    & $4$
\\ \hline
$P_n$
    & $1$
    & $q^2+1$
    & $q^4 + q^3 + 4q^2 + q + 1$
    & $q^6 + 3q^5 + 11q^4 + 10q^3 + 11q^2 + 3q + 1$
\end{tabular}.
\end{table}

In \Cref{Cor:Quadratic equation for T}, a recursive relation for the polynomials \( P_n \) is given. The first \( P_n \) displayed above are all irreducible over the rationals, but we did not find any evidence that this holds in general. It is true, however, that \( P_n \) has no rational root. The combinatorial aspects of \( P_n \) are treated separately in Section~\ref{Sec:Combinatorial interpretation} and can be used to show that, except for \( P_2 \), all coefficients of \( P_n \) are positive.

In the particular case \( q = 1 \), the spectral zeta function \( \zeta_\ZZ \) has zeros at positive integers, as follows for instance from the infinite product above.

Concerning the special values at negative integers of \( \zeta_q \), they were computed in \cite[Theorem~3]{CJKS25}; for \( q \ge 1 \) and \( m \ge 0 \),
\begin{equation}
\label{eq:CJKSformula}
\zeta_q(-m)
=
\sum_{k=0}^{m}
\binom{m}{k}^{2}\, q^{\,m-k}
\;-\;
(q-1)
\sum_{j=1}^{\lfloor m/2 \rfloor}
\;\sum_{k=0}^{m-2j}
\binom{m}{k}\binom{m}{2j+k}\,
q^{\,m-2j-k}.
\end{equation}
When viewed as functions of \( q \), these are monic polynomials of degree \( m \), with non-negative integer coefficients, although not palindromic.

At first sight, the formulas above do not suggest any direct relation between the values at positive and negative integers. In fact, \Cref{Thm:Main} will arise as a consequence of an unexpected symmetry between the special values at positive and negative integers, expressed entirely at the level of their respective generating functions. More precisely, define the generating functions of positive and negative special values by
\begin{equation}\label{Def:L and F}
    G_+(z)=\sum_{n\ge 1}\zeta_q(n)z^n,
    \qquad
    G_-(z)=\sum_{n\ge 0}\zeta_q(-n)z^n.
\end{equation}
Let \( \Omega_q^+ \ldef [(\sqrt{q}-1)^2,(\sqrt{q}+1)^2] \) be the spectrum of the combinatorial Laplacian, and let \( \Omega_q^{-} \) be the image of \( \Omega_q^+ \) under the inversion \( z \mapsto 1/z \). Writing \( \mathbb{P}^1(\CC) = \CC \cup \{\infty\} \) for the Riemann sphere, we obtain:

\begin{theorem}\label{Thm:Fundamental symmetry}
    For \( q > 1 \), the functions \( G_\pm \) admit analytic continuations to \( \mathbb{P}^1(\CC) \setminus \Omega_q^\pm \), respectively. Moreover, for \( z \) outside \( \Omega_q^+ \), the following identity holds:
    \begin{equation}
        G_+(z) + G_-\left(\frac{1}{z}\right) = 0.
    \end{equation}
\end{theorem}

This symmetry allows us to deduce \( G_+ \) from \( G_- \), the generating function of the special values \eqref{eq:CJKSformula}. In turn, \( G_- \) is itself deduced from the generating function of the probability of return after \( n \) steps on \( T_{q+1} \), already present in \cite{Kes59}. A closed formula for \( G_+ \) is therefore obtained, and \Cref{Thm:Main} follows directly.

A more structural question, raised by Karlsson in \cite[p.~9]{K20}, is whether \( \zeta_q \) admits a functional equation of the type \( s \leftrightarrow 1-s \). This problem was also raised in \cite{FK17}. We give a positive answer to this question.

\begin{theorem}\label{Thm:F.E.}
    For \( q > 1 \), define the entire function
    \[ \xi_q(s) \ldef (q-1)^s(2(q+1)\zeta_q(s)-\zeta_q(s-1)). \]
    Then, for all \( s \in \CC \),
    \[ \xi_q(1-s) = \xi_q(s). \]
\end{theorem}

As we will see, this result arises from a second symmetry at the level of the generating functions, namely \Cref{Prop:F.E. at the level of generating functions}. This also explains the somewhat unexpected normalization of \( \zeta_q \) appearing in \Cref{Thm:F.E.}.

More conceptually, this result shows that a symmetry at the level of generating functions of special values can be transferred to the zeta function itself.

\subsection*{The connection with Sato--Tate and \( q \to \infty \)}

The notation \( T_{q+1} \) is motivated by its connection with number theory and, for instance, with \( q \)-adic groups; see \cite{Ser02}. The on-diagonal spectral measure \( \mu_q \) of the adjacency operator on \( T_{q+1} \) is known as the \term{Kesten--McKay law} \cite{Kes59,McK81}, and it is absolutely continuous with respect to Lebesgue measure:
\[ \mathrm{d}\mu_q(x) = \frac{q+1}{2\pi}\frac{\sqrt{4q-x^2}}{(q+1)^2-x^2} \,\mathrm{d}x. \]
The study of \( \mu_q \) through harmonic analysis goes back to \cite{Ca72}. It is involved in the ``vertical'' Sato--Tate problem. Fix a prime \( q \); the \( q \)-Hecke operator preserves the space of cusp forms of a given level and weight. When properly rescaled, its eigenvalues asymptotically distribute according to \( \mu_q \) as the level/ weight tends to infinity; see \cite{S97}.

When \( q \to \infty \), the measure \( \mu_q \), under the rescaling \( x = \sqrt{q}\,u \), converges weakly to the \term{Sato--Tate measure}
\[
  \mu_{\infty} \ldef \frac{1}{2\pi}\sqrt{4-u^2}\,\mathrm{d}u.
\]
The Sato--Tate conjecture --- now proven in great generality \cite{BGLHT11} --- states that, under the appropriate rescaling, the Hecke eigenvalues of a fixed newform for primes \( q \le N \) asymptotically distribute according to \( \mu_{\infty} \) as \( N \to \infty \).

The Sato--Tate measure is called the Wigner semicircle law in the context of statistical physics. Both the semicircle and the Kesten--McKay laws play fundamental roles in random regular graph models, and determining precisely how \( \mu_q \) converges (locally) toward \( \mu_{\infty} \) is of particular interest in this context; see \cite{BKY17}.

From this perspective, it is natural to let \( q \to \infty \) in our setting as well. Karlsson pointed out to me that the zeta function of the Sato--Tate measure, corresponding to \( q = \infty \), has a functional equation of the type \( s \leftrightarrow 1-s \). More precisely, by defining
\[ \zeta_{\infty}(s) \ldef \frac{1}{2\pi}\int_{-2}^{2}(2-u)^{-s}\sqrt{4-u^2}\,\mathrm{d}u, \]
he observed that it admits a functional equation of the desired type.

\begin{proposition}[\cite{K25}]
    For \( \re s < 1/2 \), let
    \begin{equation}
        \xi_{\infty}(s) \ldef (2-s)\cdot 2^s\cos(\pi s/2)\zeta_{\infty}(1+s/2).
    \end{equation}
    Then \( \xi_\infty \) can be analytically continued to an entire function and the following identity holds for all \( s \in \CC \):
    \[ \xi_\infty(1-s) = \xi_\infty(s). \]
\end{proposition}

The moments of the Sato--Tate measure are Catalan numbers and, in fact, the zeta function \( \zeta_\infty \) can be viewed as an analytic continuation of the Catalan numbers. Considering the formula
\[ \zeta_\ZZ(s) = \binom{-2s}{-s}, \]
it is not surprising that \( \zeta_\ZZ \) and \( \zeta_\infty \) are related through the identity
\[ \zeta_{\infty}(1+s) = \frac{1}{1-s}\zeta_\ZZ(s). \]
This is another way of seeing that \( \zeta_\infty \) admits a functional equation.

Therefore, both \( q = 1 \) and \( q = \infty \) admit a functional equation in complete analogy with the Riemann case. Our \Cref{Thm:F.E.} completes the picture for all intermediate \( 1 < q < \infty \).
\subsection*{Further perspectives}
The method developed in the present work suggests a broader framework for studying spectral zeta functions on transitive non-amenable graphs. The fundamental symmetry at the level of generating functions of \Cref{Thm:Fundamental symmetry} relies only on the presence of a spectral gap, and therefore persists on any transitive non-amenable graph. Moreover, the functional equation of \Cref{Thm:F.E.} is a consequence of algebraic relations at the level of the generating functions. It appears that these algebraic relations are specific to the tree, and with it the closed formulas and combinatorial structure of the special values. It is known that the Green function of virtually free group is algebraic, as a consequence of Muller-Schupp theorem \cite{MS83}. For example, it would be interesting to investigate algebraic relations between generating functions of special values --- and therefore functional equations for zeta functions --- in this context, as well as computing the special values.
\subsection*{Acknowledgements}

This work was supported by the Swiss NSF Grants 200020-200400 and 200021-212864. The author warmly thanks Professor Karlsson for his many useful comments and revisions, without which this paper would not have reached its final form.

\section{The heat kernel and the spectral zeta function on \( T_{q+1} \)}\label{sec:Spectral zeta of the tree}

Let \( q \ge 1 \) be an integer. The combinatorial Laplacian on the regular tree \( T_{q+1} \) is defined by \( \Delta_q = (q+1)-A \), where \( A \) is the adjacency operator. It acts on functions \( f \colon T_{q+1} \rightarrow \CC \) by
\[
  (\Delta_q f)(x) \ldef \sum_{y \sim x}\bigl(f(x)-f(y)\bigr).
\]
It is a bounded self-adjoint operator on \( \ell^2(T_{q+1}) \) with spectrum \( \Omega_q^+ = [(\sqrt{q}-1)^2,(\sqrt{q}+1)^2] \). The heat kernel is defined as the one-parameter family of operators
\[
  K_q(t) \ldef \exp(-t\Delta_q).
\]
It acts by convolution on \( \ell^2(T_{q+1}) \) against the function
\[
  K_q(t,x,y) = \langle K_q(t)\delta_x \mid \delta_y \rangle.
\]
Precisely, for \( f \in \ell^2(T_{q+1}) \) we have
\[
  (K_q(t)f)(x) = \sum_{y \in T_{q+1}}K_q(t,x,y)f(y),
\]
and the function \( K_q(t,x,y) \) is often also called the heat kernel.

For infinite graphs such as \( T_{q+1} \), the operator \( K_q(t) \) is not trace class. To remedy this, we first observe that for transitive graphs \( \Gamma \), the function \( t \mapsto K_{\Gamma}(t,x,x) \) is independent of the vertex \( x \in \Gamma \), and second that if \( \Gamma \) is finite, then
\[
  \mathrm{Tr}\, K_{\Gamma}(t) = K_{\Gamma}(t,x,x)\cdot|\Gamma|.
\]
Therefore, for infinite transitive graphs such as \( T_{q+1} \), the trace is often replaced by the map \( t \mapsto K_{\Gamma}(t,x,x) \), which is independent of the choice of vertex \( x \in T_{q+1} \). Fix a vertex \( \mathbf{0} \in T_{q+1} \). The (local) spectral zeta function at \( \mathbf{0} \) is defined by taking the Mellin transform of this localized trace, namely
\[
  \zeta_q(s) \ldef \frac{1}{\Gamma(s)}\int_0^{+\infty}K_q(t,\mathbf{0},\mathbf{0})\,t^s\frac{\mathrm{d}t}{t}.
\]

The heat kernel on \( T_{q+1} \) has been computed in numerous places; see, for example, \cite{CY99} for a direct computation requiring no prior knowledge. Using the on-diagonal spectral measure \( \mu_q \), we have
\[
  K_q(t) \ldef K_q(t,\mathbf{0},\mathbf{0}) = \int e^{-t(q+1-x)}\,\mathrm{d}\mu_q(x).
\]
When \( q = 1 \), the tree \( T_{q+1} \) is the discrete line \( \ZZ \), and the heat kernel is given by
\[
  K_{\ZZ}(t) = e^{-2t}I_0(2t),
\]
where \( I_n \) is the \( n \)-th modified Bessel function of the first kind; see \cite{CJKS25} for a recent overview of \( K_\ZZ \).

When \( q > 1 \), the spectrum of \( \Delta_q \) does not contain \( 0 \). Therefore, for every \( s \in \CC \), the map \( x \mapsto x^{-s} \) is continuous on \( \sigma(\Delta_q)=\Omega_q^+ \), which implies that the operator \( \Delta_q^{-s} \) is well defined and bounded by the spectral theorem. From this viewpoint, the local spectral zeta function at \( \mathbf{0} \) is the entire function
\[
  \zeta_q(s) = \int (q+1-\lambda)^{-s}\,\mathrm{d}\mu_q(\lambda)
  = \langle\Delta_q^{-s}\delta_{\mathbf{0}}\mid \delta_{\mathbf{0}}\rangle.
\]

\section{The fundamental symmetry and explicit formulas for \( G_\pm \)}

The aim of this section is to prove the symmetry stated in \Cref{Thm:Fundamental symmetry}. The idea of the proof is that the expansions of the Laplace transform of the heat kernel near \( z = 0 \) and near \( z = \infty \) are directly linked with the generating functions \( G_+ \) and \( G_- \) defined in \eqref{Def:L and F}. Heuristically, \Cref{Thm:Fundamental symmetry} may be summarized by the formal identity
\[
  \sum_{n \in \ZZ}\zeta_q(n)z^n = 0.
\]
We then recall Kesten's computation of the generating function of the probability of return after \( n \) steps and use it to deduce closed formulas for \( G_\pm \).

\subsection{Proof of \Cref{Thm:Fundamental symmetry}}

The Laplace transform of the heat kernel is defined by
\[
  \mathcal{L}[K_q](s) \ldef \int_0^\infty K_q(t)e^{-st}\,\mathrm{d}t.
\]
It is analytic for \( \re s > 0 \), and in that region
\[
  \mathcal{L}[K_q](s) = \int\frac{1}{q+1-\lambda+s}\,\mathrm{d}\mu_q(\lambda),
\]
showing in fact that it is analytic outside \( -\Omega_q^+ \). Near \( s = 0 \), more precisely for \( |s| < (\sqrt{q}-1)^2 \), we have
\begin{equation}\label{Eq:Laplace near s=0}
     s\,\mathcal{L}[K_q](-s) = G_+(s),
\end{equation}
showing that the Laplace transform of the heat kernel provides the analytic continuation of \( G_+ \) to \( \CC \setminus \Omega_q^+ \). Moreover, near \( s = \infty \), more precisely when \( |s| > (\sqrt{q}+1)^2 \), we have
\begin{equation}\label{Eq:Laplace near s = infty}
    -\mathcal{L}[K_q](-s) = \frac{1}{s}G_-\left(\frac{1}{s}\right).
\end{equation}
In summary, we have shown that through the Laplace transform of the heat kernel, the functions \( G_\pm \) admit analytic continuations outside the spectral cuts \( \Omega_q^\pm \), respectively. The map \( z \mapsto 1/z \) exchanges their domains. Moreover, by combining \eqref{Eq:Laplace near s=0} and \eqref{Eq:Laplace near s = infty}, we obtain the desired relation: outside \( \Omega_q^+ \),
\[
  G_+(z) = z\,\mathcal{L}[K_q](-z) = - G_-\left(\frac{1}{z}\right),
\]
as claimed.
\hfill\(\square\)

\begin{remark}
    The proof displayed above relies only on the fact that \( \Delta_q \) is positive definite and bounded. Indeed, it is well known that the Laplace transform of the heat kernel is essentially the resolvent of the underlying operator. The same argument applies to the resolvent of any such operator.

    Consequently, for any transitive graph of finite degree with a spectral gap, this symmetry applies.
\end{remark}

\subsection{Kesten's generating function and a formula for \( G_- \)}

The Kesten--McKay law \( \mu_q \), \ie\ the spectral measure of the adjacency operator \( A \) at the root, was computed by Kesten in \cite{Kes59} by studying the probability of return after \( n \) steps, or equivalently the numbers
\[
c_q(n) \ldef \langle A^{n}\delta_{\mathbf{0}} \mid \delta_{\mathbf{0}}\rangle
\]
of closed walks of length \( n \) from the root. These are the moments of \(\mu_q\):
\[
c_q(n)=\int \lambda^{n}\, \mathrm{d}\mu_q(\lambda).
\]
Kesten computed their generating function
\[
  F(z)
  \ldef \sum_{n\ge 0} c_q(n) z^n
  = \frac{1}{2}\,\frac{(q+1)\sqrt{1-4q z^2} - (q-1)}{1-(q+1)^2 z^2},
\]
where the branch is chosen so that \( F(0) = 1 \).
Since \( \Delta_q = (q+1)-A \), the functions \( F \) and \( G_- \) satisfy the relation
\[
  G_-(z) = \frac{1}{1-(q+1)z}F\left(\frac{z}{(q+1)z-1}\right),
\]
and, at the level of the coefficients,
\[
  \zeta_q(-m) = \sum_{j=0}^{m}\binom{m}{j}(-1)^j c_q(j)(q+1)^{m-j}.
\]
The formula for \( F \) implies that the \( c_q(n) \) are polynomials in \( q \) with integer coefficients, which in turn implies that for negative integers \(-m \le 0\), the special values \( \zeta_q(-m) \) are polynomials in \( q \) with integer coefficients. Using the relation between \( G_- \) and \( F \), we obtain
\begin{equation}\label{Eq:elementary expression for F}
    G_-(z) = \frac{1}{2}\,\frac{
(q+1)\sqrt{1-2(q+1)z + (q-1)^2 z^2}
+ z(q^2-1) - (q-1)
}{
1 - 2(q+1)z
}.
\end{equation}
Extracting from this expression a closed formula for the special values at negative integers is cumbersome. However, using the hypergeometric nature of \( \zeta_q \) and a completely different method, such a closed formula was obtained in \cite{CJKS25}; see \eqref{eq:CJKSformula} in the introduction.

\subsection{The formula for \( G_+ \)}

By combining \eqref{Eq:elementary expression for F} and \Cref{Thm:Fundamental symmetry}, we obtain an elementary expression for \( G_+ \), the generating function of the positive special values of \( \zeta_q \).

\begin{lemma}\label{Lem:Elementary expression for L}
For \( q > 1 \), the following equivalent statements hold:
\begin{enumerate}[label = (\roman*)]
  \item For \( z \in \CC \setminus \Omega_q^+ \), we have
  \begin{equation}\label{Lem:Alebraic form of L}
        G_+(z)=\frac{1}{2}\,
        \frac{(q+1)\sqrt{(q-1)^2-2z(q+1)+z^2}+z(q-1)-(q^2-1)}{z-2(q+1)}.
  \end{equation}
  \item The generating function \( G_+ \) satisfies the quadratic equation
  \begin{equation}\label{Lem:Quadratic equation of L}
      (2(q+1)-z)G_+^2+(q-1)(z-(q+1))G_++zq=0.
  \end{equation}
  \item The sequence of real numbers defined by \( a_n \ldef \zeta_q(n) \) satisfies, for \( n \ge 2 \),
  \begin{equation}\label{Lem:Recursive relation for a_n}
      (q^2-1)a_n
  =2(q+1)\sum_{j=1}^{n-1}a_ja_{n-j}
  -\sum_{j=1}^{n-2}a_ja_{n-1-j}
  +(q-1)a_{n-1},
  \end{equation}
  with \( a_0 = 1 \) and \( a_1 = \tfrac{q}{q^2-1} \).
\end{enumerate}
\end{lemma}

\begin{proof}
    The first statement follows by a direct substitution. The only delicate point is the choice of branch for the square root, but the condition \( G_+(0) = 0 \) determines it uniquely. The remaining statements are equivalent reformulations of \eqref{Lem:Alebraic form of L} and follow directly.
\end{proof}

\section{Proof of \Cref{Thm:Main} and a recursive relation for the \( P_n \)'s}\label{Sec:4}

In this section, using the explicit expression for \( G_+ \) obtained in \Cref{Lem:Elementary expression for L}, we prove the first part of \Cref{Thm:Main}, leaving the combinatorial interpretation to the final section. We then translate the recursive relation \eqref{Lem:Recursive relation for a_n} into a recursive relation for the \( P_n \)'s.

\subsection{Proof of \Cref{Thm:Main}}

By \eqref{Lem:Alebraic form of L}, we write
\[
  G_+(z(q+1)) = \frac{q-1}{2}+ \frac{q-1}{2(z-2)}\left[1+\sqrt{1+k^2z(z-2)}\right],
\]
with \( k = \tfrac{q+1}{q-1} \). Observe that \( k^2 \) is invariant under the inversion \( q \leftrightarrow 1/q \). Expanding the square root near \( z = 0 \), one obtains that for \( n \ge 1 \),
\[
  (q+1)^n\zeta_q(n) = (q-1)R_n(k^2),
\]
where \( R_n \) is a polynomial with rational coefficients of degree \( n \) satisfying \( R_n(1) = 0 \). Translating back in terms of \( q \), we write
\[
  R_n(k^2) = \frac{1}{(q-1)^{2n}}\,qP_n(q),
\]
where \( P_n(q) \) is a rational polynomial satisfying \( q^{2n-2}P_n(1/q) = P_n(q) \).

Now recall that the spectral radius of \( A \) is \( 2\sqrt{q} \), implying that for fixed \( s \in \CC \), the convergence in operator norm
\[
  \|(q+1)^s\Delta_q^{-s}- I_q\| \underset{q \to \infty}{\longrightarrow} 0
\]
holds, where \( I_q \) is the identity operator on \( \ell^2(T_{q+1}) \). Thus
\[
  (q+1)^s\zeta_q(s) \underset{q \to \infty}{\longrightarrow} 1
\]
for every \( s \in \CC \). It follows that \( P_n \) is monic of degree \( 2n-2 \). In summary, we have shown that for \( n \ge 1 \),
\begin{equation}\label{Formula:definition of P_n}
     \zeta_q(n) = \frac{q}{(q-1)^{2n-1}(q+1)^n}P_n(q),
\end{equation}
with \( P_n \) a palindromic and monic polynomial of degree \( 2n-2 \).

To conclude the proof, it remains to show that the coefficients of \( P_n \) are integers. In view of \eqref{Formula:definition of P_n}, this is equivalent to showing that for all \( n \ge 0 \), \( \zeta_q(n) \) belongs to the ring \( \ZZ[[q]] \). This follows by induction from the recursive relation \eqref{Lem:Recursive relation for a_n}.
\hfill\(\square\)

\begin{remark}
The symmetry exhibited by the palindromic property of the polynomials \( P_n \) in \Cref{Thm:Main} is already present at the level of the spectral measure \( \mu_q \). Indeed, through the change of variable \( \lambda = (q+1)\, x \) in \eqref{Def:zeta function on the tree}, namely by considering the normalized Laplacian instead of the combinatorial one, the function \( \zeta_q \) can be written as
\begin{equation}\label{Eq:zeta_q as function of the spectral radius}
    (q+1)^{s-1}\zeta_q(s) = \frac{1}{2\pi}\int_{-\rho}^{\rho}(1-x)^{-s}\frac{\sqrt{\rho^2-x^2}}{1-x^2}\,\mathrm{d}x,
\end{equation}
where \( \rho \) is the spectral radius of the transition operator of the random walk on \( T_{q+1} \),
\[ \rho^{-1} = \frac{1}{2}\left(\sqrt{q} + \frac{1}{\sqrt{q}}\right). \]
However, it is curious that, although the right-hand side of \eqref{Eq:zeta_q as function of the spectral radius} is invariant under \( q \leftrightarrow 1/q \), the left-hand side is not strictly invariant under this inversion. For example, the special values at negative integers \eqref{eq:CJKSformula} do not have any palindromic structure. Even at positive integers, the special values change sign under \( q \leftrightarrow 1/q \). This breaking of symmetry comes from the fact that for fixed \( s \in \CC \), the right-hand side of \eqref{Eq:zeta_q as function of the spectral radius} --- when viewed as a function of \( q \) --- is not analytic at \( q = 1 \). The value \( q = 1 \) is the only value for which \( 0 \) belongs to the spectrum of \( \Delta_q \).

We also note that the algebraic structure of the special values in \Cref{Thm:Main} ultimately relies on the algebraic nature of \( G_+ \). We deduce it from the fundamental symmetry between the generating functions of special values, together with Kesten's previous work. This derivation can also be achieved independently through the \term{spectral angle representation} of \( \zeta_q \):
\[
  \zeta_q(s) = \frac{2}{\pi}q(q+1)\int_0^\pi(q+1-2\sqrt{q}\cos\theta)^{-s}\frac{\sin^2 \theta}{(q+1)^2-4q\cos^2\theta}\,\mathrm{d}\theta.
\]
By directly computing the generating function and the resulting integral, one obtains \Cref{Lem:Elementary expression for L} by an alternative method.

However, we believe that our approach is more conceptual and technically less demanding. The symmetry of \Cref{Thm:Fundamental symmetry} is the structural mechanism that explains why the generating function \( G_+ \) admits an algebraic closed form and, consequently, why the polynomials \( P_n \) inherit palindromic and integrality properties. By contrast, the angle parametrization leads to a longer and more technical derivation, and makes the conceptual origin of the algebraic expression less transparent.
\end{remark}

\subsection{A recursive relation for the \( P_n \)'s}

The first values of \( P_n \) can be computed by hand:
\begin{align*}
    P_1(q) & = 1, \\
    P_2(q) & = q^2+1, \\
    P_3(q) & = q^4 + q^3 + 4q^2 + q + 1, \\
    P_4(q) & = q^6 + 3q^5 + 11q^4 + 10q^3 + 11q^2 + 3q + 1, \\
    P_5(q) & = q^8 + 6q^7 + 26q^6 + 46q^5 + 66q^4 + 46q^3 + 26q^2 + 6q + 1.
\end{align*}
Remarkably, all coefficients appearing above are non-negative. In fact, this phenomenon holds for every polynomial \( P_n \), as will be proved in the next section, where we give a combinatorial interpretation of the coefficients. Before doing so, we translate \Cref{Lem:Elementary expression for L} to obtain a recursive relation for the \( P_n \)'s and a quadratic equation for their generating function.

\begin{corollary}[of \Cref{Lem:Elementary expression for L}]\label{Cor:Quadratic equation for T}
    Let \( T \) be the generating function of the family of polynomials \( P_n \), namely
    \[
      T(z) \ldef \sum_{n \ge 1}P_n(q)z^{n-1}.
    \]
    It satisfies the quadratic equation
    \[
      qz\bigl(2-z(q-1)^2\bigr)T^2 +\bigl(z(q-1)^2-1\bigr)T + 1 = 0,
    \]
    and consequently the family \( P_n \) satisfies the recursive relation, for \( n \ge 1 \),
    \begin{equation}
        P_{n+1} = 2q\sum_{j=1}^nP_jP_{n+1-j} -q(q-1)^2\sum_{j=1}^{n-1}P_jP_{n-j}+(q-1)^2P_n,
    \end{equation}
    with \( P_0 = 1 \).
\end{corollary}

Observe that, by working directly with the recursive relation, one can also deduce \Cref{Thm:Main}.

\section{The functional equation}

In this section we prove \Cref{Thm:F.E.}.

The first step toward this result is yet another relation between the generating functions \( G_\pm \)---defined in the introduction---obtained by taking advantage of their algebraic expressions.

\begin{proposition}\label{Prop:F.E. at the level of generating functions}
    Let \( q > 1 \), \( k \ldef \tfrac{q+1}{q-1} \), and define
    \[ \mathcal{E}(z) \ldef G_-\left(\frac{z}{q-1}\right)\cdot(1-2kz) + G_+((q-1)z)\cdot(2k-z). \]
    Then \( \mathcal{E} \) can be analytically continued to an entire function and satisfies
    \begin{equation}\label{Eq:Symmetry of G_pm}
        \mathcal{E}(z) = z+1.
    \end{equation}
\end{proposition}

\begin{proof}
    From \eqref{Eq:elementary expression for F} and \eqref{Lem:Alebraic form of L}, around \( z = 0 \), we have
    \[ G_+((q-1)z) = \frac{1}{2}\frac{(q+1)\sqrt{1-2zk+z^2}+z(q-1)-(q+1)}{z-2k}, \]
    and
    \[ G_-\left(\frac{z}{q-1}\right) = \frac{1}{2}\frac{(q+1)\sqrt{1-2kz+z^2}+(q+1)z-(q-1)}{1-2kz}, \]
    which shows that, for \( z \) small enough,
    \[ \mathcal{E}(z) = z + 1. \]
\end{proof}

Consequently, by comparing coefficients in the generating series, we obtain the following recursive relation.

\begin{corollary}
    Let \( q > 1 \) and \( a_n \ldef \zeta_q(n) \) for \( n \in \ZZ \). Then for any \( n \in \ZZ \), we have the two-step recursive relation
    \begin{equation}\label{Eq:F.E. at special values}
        a_{-n}-2(q+1)a_{1-n} = (q-1)^{2n-1}\left[a_{n-1}-2(q+1)a_n\right].
    \end{equation}
\end{corollary}

Heuristically, the symmetry \eqref{Eq:Symmetry of G_pm} at the level of generating functions should induce a corresponding symmetry at the level of the zeta function. This is consistent with the general principle that the generating function, the heat kernel, and the spectral zeta function are related through Laplace -- and Mellin -- transform.

One is therefore led to consider the Mellin transform of \eqref{Eq:Symmetry of G_pm}, with the aim of recovering an analytic continuation of \eqref{Eq:F.E. at special values}, and hence the desired functional equation. The main difficulty is that the spectral cuts lie on the natural path of integration. A Hankel contour could be used to bypass the cut, but in the presence of a spectral gap one can choose a much more convenient contour.

Recall that for \( B \) a self-adjoint operator with spectrum \( \sigma(B) \), the resolvent of \( B \) is defined to be the holomorphic map
\[ \sigma(B) \not \ni z \longmapsto R_B(z) \ldef \frac{1}{z-B}. \]

\begin{lemma}\label{Lem:Cauchy integral formula}
    Let \( B \) be a bounded self-adjoint operator with spectrum \( \sigma(B) \). Let \( U \subset \CC \) be an open set containing \( \sigma(B) \), and let \( \gamma \) be a piecewise smooth positively oriented Jordan curve inside \( U \), containing \( \sigma(B) \) in its interior. Then for any holomorphic function \( h \colon U \rightarrow \CC \) and any \( n \in \ZZ \), we have
    \[ \frac{1}{2\pi i}\int_\gamma h(z)R_B(z)^n\,\mathrm{d}z = \begin{cases}
        h(B) & \text{if } n = 1, \\
        0 & \text{otherwise.}
    \end{cases} \]
\end{lemma}

\begin{proof}
    This is a straightforward application of the spectral theorem and its functional calculus. Indeed, let \( \mathbb{P}_B \) be the projection-valued measure of \( B \) provided by the spectral theorem. Then
    \begin{align*}
        \frac{1}{2\pi i}\int_\gamma h(z)R_B(z)^n\,\mathrm{d}z
        &= \frac{1}{2\pi i}\int_\gamma \int_{\sigma(B)}\frac{h(z)}{(z-\lambda)^n}\,\mathrm{d}\mathbb{P}_B(\lambda)\,\mathrm{d}z \\
        &= \int_{\sigma(B)}\frac{1}{2\pi i}\int_\gamma\frac{h(z)}{(z-\lambda)^n}\,\mathrm{d}z\,\mathrm{d}\mathbb{P}_B(\lambda).
    \end{align*}
    By Cauchy's integral formula, the inner integral is \( 0 \) if \( n \neq 1 \), and \( h(\lambda) \) otherwise. Therefore, for \( n = 1 \), we obtain
    \[ \frac{1}{2 \pi i}\int_\gamma h(z)R_B(z)\,\mathrm{d}z = \int_{\sigma(B)}h(\lambda)\,\mathrm{d}\mathbb{P}_B(\lambda) = h(B). \qedhere \]
\end{proof}

\begin{proof}[Proof of \Cref{Thm:F.E.}]
    Consider the operator
    \[ B \ldef \frac{\Delta_q}{q-1}, \]
    and observe that its spectrum is \( \sigma(B) = [\beta,\beta^{-1}] \), with
    \[ \beta \ldef \frac{\sqrt{q}-1}{\sqrt{q}+1}>0. \]
    In particular, \( B \) has a spectral gap and its inverse \( B^{-1} \) has the same spectrum as \( B \). Recall also that
    \[ \xi_q(s) \ldef (q-1)^s(2(q+1)\zeta_q(s)-\zeta_q(s-1)) = (q-1)\langle (2kB^{-s}-B^{1-s})\delta_\mathbf{0}\mid \delta_\mathbf{0} \rangle, \]
    where \( k = \tfrac{q+1}{q-1} \).
    From the definition of \( G_\pm \), we have
    \begin{align*}
        G_+((q-1)z) &= \langle -zR_B(z)\delta_\mathbf{0} \mid \delta_\mathbf{0} \rangle, \\
        G_-\left(\frac{z}{q-1}\right) &= \langle (1-zR_{B^{-1}}(z))\delta_\mathbf{0} \mid \delta_\mathbf{0} \rangle.
    \end{align*}
    Take \( \gamma \) to be a positively oriented parametrization of the circle:

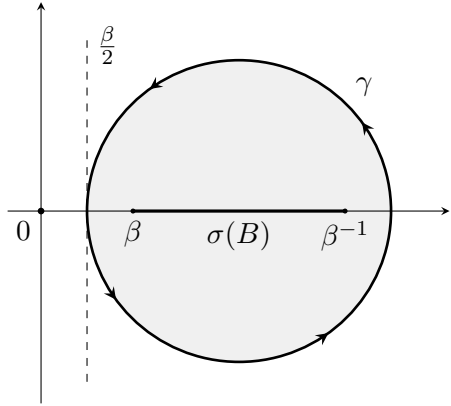
\begin{center}

\begin{tikzpicture}[scale=2.2,>=stealth]

\def\betaval{0.55}
\pgfmathsetmacro{\invbeta}{1/\betaval}
\pgfmathsetmacro{\leftend}{\betaval/2}
\pgfmathsetmacro{\rightend}{\invbeta + \betaval/2}
\pgfmathsetmacro{\kcent}{(\leftend+\rightend)/2}   
\pgfmathsetmacro{\R}{(\rightend-\leftend)/2}       

\draw[->] (-0.2,0) -- (\rightend+0.35,0) node[right] {};
\draw[->] (0,-\R-0.25) -- (0,\R+0.35) node[above] {};

\fill (0,0) circle (0.02) node[below left] {$0$};

\draw[line width=1.2pt] (\betaval,0) -- (\invbeta,0);
\node[below] at ({(\betaval+\invbeta)/2},0) {$\sigma(B)$};
\fill (\betaval,0) circle (0.015) node[below] {$\beta$};
\fill (\invbeta,0) circle (0.015) node[below] {$\beta^{-1}$};

\fill (\leftend,1) circle (0.0012) node[right] {$\frac{\beta}{2}$};


\draw[
  line width=1.0pt,
  decoration={markings,
    mark=at position 0.10 with {\arrow{>}},
    mark=at position 0.35 with {\arrow{>}},
    mark=at position 0.60 with {\arrow{>}},
    mark=at position 0.85 with {\arrow{>}}
  },
  postaction={decorate}
]
(\kcent,0) circle (\R);

\node[above right] at ({\kcent+0.7*\R},{0.7*\R}) {$\gamma$};

\fill[opacity=0.06] (\kcent,0) circle (\R);

\draw[dashed] (\leftend,-\R-0.12) -- (\leftend,\R+0.12);

\end{tikzpicture}

\end{center}
Now for \( s \in \CC \), under the standard branch of the logarithm, the map \( z \mapsto z^{-s} \) is holomorphic on the open set \( U \ldef \{ z \in \CC \mid \re(z) > 0 \} \). By combining \Cref{Lem:Cauchy integral formula} and \Cref{Prop:F.E. at the level of generating functions}, we obtain, since \( \mathcal{E} \) is entire,
\begin{align*}
    0
    &= \frac{1}{2\pi i}\int_\gamma \mathcal{E}(z)z^{-s}\frac{\mathrm{d}z}{z} \\
    &= \Big\langle\frac{1}{2\pi i}\int_\gamma \left[(1-2kz)(1-zR_{B^{-1}}(z)) + (z-2k)zR_B(z)\right]z^{-s}\frac{\mathrm{d}z}{z}\,\delta_\mathbf{0} \mid \delta_\mathbf{0}\Big\rangle \\
    &= \langle (-B^{s} +2kB^{s-1} + B^{1-s}-2kB^{-s})\delta_\mathbf{0} \mid \delta_\mathbf{0}\rangle \\
    &= \frac{1}{q-1}(\xi_q(1-s) - \xi_q(s)). \qedhere
\end{align*} 
\end{proof}

\section{Combinatorial interpretation of \( P_n \)}\label{Sec:Combinatorial interpretation}

In this section we show that the coefficients of \( P_n \) admit a combinatorial interpretation in terms of $2$-coloured Dyck words of length \( 2n-2 \). The first hint in this direction is the observation that by substituting \( q = 1 \) into the quadratic equation satisfied by the generating function \( T \) in \Cref{Cor:Quadratic equation for T}, we obtain
\[ \sum_{n \ge 1}P_n(1)z^n = \frac{1-\sqrt{1-8z}}{4z}, \]
the generating function of the numbers \( 2^nC_n \), where \( C_n \) is the \( n \)-th Catalan number. It is well known that \( C_n \) counts the number of \term{Dyck paths} of length \( 2n \). A Dyck path is a path starting at \( (0,0) \in \ZZ^2 \) that ends on --- and always stays above --- the horizontal axis, with allowed steps in the set \( \{(1,1),(1,-1)\} \); see \Cref{fig:dyck8} below.

\begin{figure}[ht]
\centering
\begin{tikzpicture}[x=0.75cm,y=0.75cm, line cap=round, line join=round]

\draw[->] (-0.2,0) -- (8.6,0) node[right] {$x$};
\draw[->] (0,-0.2) -- (0,4.6) node[above] {$y$};

\draw[very thick]
(0,0)--(1,1)--(2,2)--(3,1)--(4,2)--(5,1)--(6,0)--(7,1)--(8,0);

\fill (0,0) circle (2pt);
\fill (8,0) circle (2pt);

\end{tikzpicture}
\caption{Example of a Dyck path of length $8$.}
\label{fig:dyck8}
\end{figure}
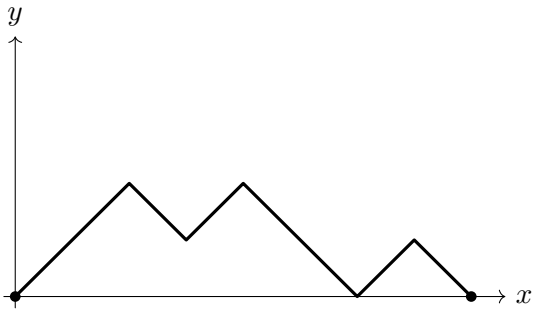

A $2$-\term{coloured Dyck path} is a Dyck path in which the downward steps \( (1,-1) \) are coloured in two colours; see \Cref{fig:dyck12} below.

\begin{figure}[ht]
\centering
\begin{tikzpicture}[x=0.6cm,y=0.6cm, line cap=round, line join=round]

\draw[->] (-0.2,0) -- (12.6,0) node[right] {$x$};

\draw[very thick] (0,0) -- (1,1);
\draw[very thick] (1,1) -- (2,2);
\draw[very thick] (2,2) -- (3,3);

\draw[very thick, blue]  (3,3) -- (4,2);

\draw[very thick] (4,2) -- (5,3);
\draw[very thick] (5,3) -- (6,4);

\draw[very thick, blue] (6,4) -- (7,3);
\draw[very thick, red]  (7,3) -- (8,2);

\draw[very thick] (8,2) -- (9,3);

\draw[very thick, blue] (9,3) -- (10,2);
\draw[very thick, red]  (10,2) -- (11,1);
\draw[very thick, blue] (11,1) -- (12,0);

\fill (0,0) circle (2pt);
\fill (12,0) circle (2pt);

\node[anchor=west] at (0,4.3) {\small $UUUBUUBRUBRB$};

\end{tikzpicture}
\caption{Example of a $2$-coloured Dyck path of length $12$.}
\label{fig:dyck12}
\end{figure}
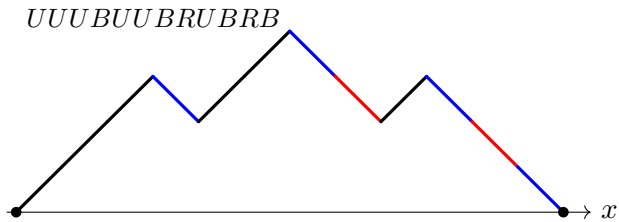

The identity \( P_n(1) = 2^{n-1}C_{n-1} \) strongly suggests that the coefficients of \( P_n \) count different types of $2$-coloured Dyck paths.

\begin{definition}
    Let \( \Sigma^*_3 \) be the free monoid on the set \( \{U,B,R\} \). To each word \( w \in \Sigma^*_3 \) we associate a \term{coloured path}, read from left to right, where \( U \) corresponds to the upward movement \( (1,1) \), while \( R \) and \( B \) correspond to downward movements coloured in red and blue, respectively. The set of $2$-\term{coloured Dyck words} is the subset \( \mathbb{D} \subset \Sigma^*_3 \) of those elements \( w \) such that the corresponding coloured path is a $2$-coloured Dyck path. The number of $2$-coloured Dyck words of length \( 2n \) is \( 2^nC_n \).

    Let \( w \in \Sigma_3^* \). A \term{block of colour} \( \beta \in \{U,B,R\} \) in \( w \) is a maximal and non-empty subword of \( w \) composed solely of the letter \( \beta \). Denote by \( r_\beta(w) \) the number of blocks of colour \( \beta \) in \( w \).
\end{definition}

We define a weight function on $2$-coloured Dyck words by taking the number of \( U \)'s, adding the number of blue blocks, and subtracting the number of red blocks. Precisely, for \( w \in \mathbb{D} \) of length \( 2n \), let
\[ h(w) \ldef n + r_B(w) - r_R(w). \]
Next we define the polynomials associated with \( h \): for \( n \ge 0 \),
\[ Q_n(t) = \sum_{w \in \mathbb{D}_n}t^{h(w)} = \sum_{k = 0}^{2n}\#\{w \in \mathbb{D}_n \mid h(w) = k\} t^k, \]
where \( \mathbb{D}_n \) is the set of \( w \in \mathbb{D} \) of length \( 2n \).
It is straightforward to see that the \( Q_n \) are monic, palindromic, and have non-negative integer coefficients. The first values of \( Q_n \) can be computed by hand:
\begin{align*}
    Q_0(t) & = 1, \\
    Q_1(t) & = t^2+1, \\
    Q_2(t) & = t^4 + t^3 + 4t^2 + t + 1, \\
    Q_3(t) & = t^6 + 3t^5 + 11t^4 + 10t^3 + 11t^2 + 3t + 1, \\
    Q_4(t) & = t^8 + 6t^7 + 26t^6 + 46t^5 + 66t^4 + 46t^3 + 26t^2 + 6t + 1.
\end{align*}
Remarkably, for all \( n = 0,\ldots,4 \), we have \( Q_n = P_{n+1} \); in fact, we shall show that this always holds.

\begin{theorem}\label{Thm:Combinatorial expression of P_n}
    For each \( n \ge 0 \), we have \( Q_n = P_{n+1} \).
\end{theorem}

The proof rests on the fact that the weight \( h \) behaves well recursively. It yields a quadratic equation for the generating function of the \( Q_n \)'s which coincides with the one for the generating function \( T \) of the \( P_n \)'s, as described in \Cref{Cor:Quadratic equation for T}.

\begin{proof}
    The generating function \( H(t,z) \ldef \sum_{n \ge 0}Q_n(t)z^n \) can be written as
    \[
      H = \sum_{w \in \mathbb{D}}t^{h(w)}z^{l(w)},
    \]
    where \( l(w) \) is half the length of \( w \). Let us define the auxiliary generating functions
    \[
      E_B \ldef \sum_{w \in \mathbb{D}(B)}t^{h(w)}z^{l(w)}, \quad \text{and} \quad
      E_{R} \ldef \sum_{w \in \mathbb{D}(R)}t^{h(w)}z^{l(w)},
    \]
    where \( \mathbb{D}(\beta) \) denotes the subset of words in \( \mathbb{D} \) that end with \( \beta \). Then clearly,
    \begin{equation}\label{Eq:Decomposition of H}
        H = 1 + E_B + E_{R}.
    \end{equation}
    Let us investigate the recursive behaviour of \( h \). Observe first that every non-empty \( w \in \mathbb{D} \) can be written uniquely as
    \[ w_1Uw_2\beta, \]
    where \( w_1,w_2 \in \mathbb{D} \) (possibly empty), and \( \beta \in \{B,R\} \). Then
    \[ h(w_1Uw_2\beta) =
        h(w_1)+1+h(w_2) +\varepsilon(w_2\beta), \]
    where
    \[ \varepsilon(w_2\beta) = \begin{cases}
        0 & \text{if } w_2 \text{ ends with the letter } \beta, \\
        \mathrm{sign}(\beta) & \text{otherwise,}
    \end{cases} \]
    and where \( \mathrm{sign}(B) = 1 \) and \( \mathrm{sign}(R) = -1 \).
    Now let us compute
    \begin{align*}
        E_B
        &= \sum_{w_1Uw_2B}t^{h(w_1)+1+h(w_2)+\varepsilon(w_2B)}z^{l(w_1)+l(w_2)+1} \\
        &= tz \cdot H \cdot \sum_{w_2 \in \mathbb{D}}t^{h(w_2)+\varepsilon(w_2B)}z^{l(w_2)} \\
        &= tz \cdot H \cdot (E_B +t(1+E_R)).
    \end{align*}
    Similarly, we obtain
    \[
      E_{R}= H \cdot tz \cdot (E_{R}+t^{-1}(1+E_B)).
    \]
    By isolating \( E_B \) and \( E_{R} \) in each equation and by setting
    \[
      u \ldef \frac{H\cdot tz}{1-H\cdot tz},
    \]
    we obtain
    \[
      E_B = \frac{u(u+t)}{1-u^2}, \quad \text{and} \quad
      E_{R} = \frac{u(u+t^{-1})}{1-u^2}.
    \]
    Finally, substituting into \eqref{Eq:Decomposition of H}, we can write
    \begin{align*}
        H
        &= 1 + E_B + E_{R} \\
        &= \frac{(u+t)(u+t^{-1})}{1-u^2} \\
        &= \frac{-tz^2(t-1)^2H^2+z(t-1)^2H+1}{1-2tzH}.
    \end{align*}
    Consequently, \( H \) satisfies the quadratic equation
    \begin{equation}
        tz\bigl(2-z(t-1)^2\bigr)H^2 +\bigl(z(t-1)^2-1\bigr)H + 1 = 0.
    \end{equation}
    This is the same quadratic equation as the one satisfied by \( T \); see \Cref{Cor:Quadratic equation for T}. Since both satisfy \( T(0) = 1 = H(0) \), we deduce that \( H = T \), and consequently the result follows.
\end{proof}

As an immediate consequence, the coefficients of \( P_n \) are non-negative. In fact, \Cref{Thm:Main} is a corollary of \Cref{Thm:Combinatorial expression of P_n}. Using this combinatorial interpretation, it can be shown that, except for \( P_2 \), all the coefficients of \( P_n \) are positive.

D.M.
Section de mathématiques, Université de Genève, rue du Conseil-Général 7-9, 1205 Genève, Suisse,
dylan.mueller@unige.ch 

\end{document}